\documentclass[12pt]{article}

\usepackage[english]{babel}
\usepackage{graphicx} 

\usepackage{comment}
\usepackage{subcaption}
\usepackage{comment}
\usepackage{amsmath}
\usepackage{amssymb}
\usepackage{cases}
\usepackage[title]{appendix}
\usepackage[colorlinks=true, allcolors=blue]{hyperref}
\DeclareMathOperator{\R}{\mathbb{R}}

\DeclareMathOperator{\Sp}
{\mathbb{S}}

\DeclareMathOperator{\FLs}{(-\Delta)^s}

\DeclareMathOperator{\supp}{supp}

\usepackage{amsthm}
\newtheorem{theorem}{Theorem}[section]
\newtheorem{lemma}[theorem]{Lemma}
\newtheorem{remark}[theorem]{Remark}

\newtheorem{assumption}[theorem]{Assumption}
\theoremstyle{definition}
\newtheorem{definition}{Definition}[section]
\newcommand*\hgeom[2]{{}_{#1}F_{#2}}
\usepackage[letterpaper,top=2cm,bottom=2cm,left=2cm,right=2cm,marginparwidth=1.75cm]{geometry}

\usepackage{amsmath}
\usepackage{graphicx}
\usepackage[colorlinks=true, allcolors=blue]{hyperref}
\usepackage{xurl}
\title{Reconstruction of the Support of an Inhomogeneity for the Fractional Helmholtz Equation}
\author{Dana Zilberberg\footnote{Department of Mathematics, Rutgers University, New Brunswick, NJ, USA (dana.zilberberg@gmail.com)} }

\begin{document}
\maketitle
\fontsize{12}{18}\selectfont
\begin{abstract}
    We consider the inverse scattering problem for inhomogeneous media of compact support governed by the fractional  s-Helmholtz equation, with  $0<s<1$, in dimensions $d=1,2,3$. In particular, we study the determination of the support of the inhomogeneity from the far-field pattern of the scattered field generated by plane waves for all incident directions at a fixed frequency. The far-field pattern is defined as the principal term in the asymptotic expansion of the scattered field at infinity. It is shown in \cite{zilberberg2026limiting}  that, up to a multiplicative constant, this coincides with the far-field pattern corresponding to the classical Helmholtz equation with the same inhomogeneity. Our approach is based on the development of the factorization method, which not only leads to an efficient and easily implementable reconstruction algorithm, but also provides a uniqueness result for determining the support of an admissible set of inhomogeneities. A fundamental ingredient in the analysis is a new transmission eigenvalue problem, whose eigenvalues must be excluded. Therefore, we prove that they are discrete with no finite accumulation points. We present numerical examples in dimension $d=2$  for both the direct and inverse problems.
\end{abstract}
\section{Introduction}
The Helmholtz equation arises naturally as the time-harmonic reduction of the wave equation and governs acoustic scattering phenomena. Recently, considerable attention has been directed toward fractional generalizations of classical models, in particular those involving the fractional Laplacian, which is intrinsically nonlocal and admits several equivalent formulations (see \cite{garofalo2017fractional} for a basic review of this operator). In this work we consider the fractional Helmholtz equation wherein the classical Laplacian is replaced by the fractional Laplacian of order $0<s<1$, in dimension $d = 1,2,3$. More precisely, we are concerned with the scattering problem by a compactly supported inhomogeneous medium, which is studied in detail in \cite{zilberberg2026limiting}. This problem is modeled in terms of the scattered field $u$ such that
\begin{align}\label{inhomogeneous_fractional_Helmholtz}
\FLs u - k^{2s}n u = k^{2s}(n-1)u^{inc} \mbox{ in } \R^d
\end{align}
where the positive real-valued function $n\in L^\infty(\R^d)$ (referred to as the refractive index) is such that $n-1$ has compact support. We denote by $\overline D:=\supp(n-1)$. Here, $u^{inc}$ is the probing wave which solves
\begin{equation}\label{inci}
\FLs u^{inc} - k^{2s} u^{inc}=0.
\end{equation}
To guarantee the well-posedness, $u$ must satisfy the {\em{Sommerfeld Radiation Condition}} (SRC)
\begin{align}\label{SRC}\tag{SRC}
\lim_{|x|\to \infty} |x|^{\frac{d-1}{2}} \left( \frac{\partial u( x)}{\partial |x|} - ik u(x) \right) = 0 \quad \mbox{ uniformly in } \hat x := \frac{x}{|x|}.
\end{align}

The fractional Helmholtz equation is used to describe wave propagation in complex, attenuating media, or media with nonlocal properties that cannot be accurately represented by the classical Helmholtz equation, such as wave propagation in lossy media and in complex geological formations, particularly in the context of nonlocal elasticity. It provides a framework for understanding and simulating wave propagation in fractal or inhomogeneous materials where standard integer-order derivatives are insufficient. Special choices of perturbations of the fractional Helmholtz operator include the relativistic Schrödinger operator and the anomalous transport operator. The main goal of this paper is to study the corresponding inverse problem. Inverse problems for fractional partial differential equations have drawn significant attention in recent years (see, for instance, the monograph \cite{KR23}). Most results for models with fractional derivatives in the spatial variable in the Euclidean setting address interior problems, in which Dirichlet or Neumann data for the fractional Laplacian are prescribed on the exterior of a bounded domain (we refer the reader to \cite{GSU20, FGKRSU25} and the references therein). Owing to the strong unique continuation property associated with nonlocal operators, such inverse problems often yield stronger uniqueness and stability results than their local counterparts. Most recently, uniqueness results for the inverse problem of determining the potential (or its properties) for the fractional Schrödinger equation from scattering data have been obtained in \cite{das2025inverse, LNOS25, UW25}.

For the scattering problem described above, we measure the far-field scattering data, i.e., the far-field pattern $u^\infty(\hat x, \theta)$, $\hat x \in {\mathbb S}^{d-1}$, of the scattered wave $u(x, \theta)$ due to plane wave incident fields $u^{inc}(x):=e^{ikx\cdot \theta}$, for all directions $\theta \in {\mathbb S}^{d-1}$, which solve (\ref{inci}). A similar far-field scattering configuration for the fractional Schrödinger equation in the whole space is considered in \cite{das2025inverse}, where the authors show that a compactly supported potential is uniquely determined from the knowledge of the far-field patterns $u^\infty(\hat x, \theta)$, and for an infinite sequence of high frequencies $k$ tending to $+\infty$. In this paper, we aim to determine the support $D$ of the contrast $n-1$ from the far-field scattering data $u^\infty(\hat x, \theta)$, $\hat x,\theta \in {\mathbb S}^{d-1}$, at a fixed frequency that does not belong to a discrete set of exceptional frequencies without finite accumulation points.

Our approach is based on the factorization method, which belongs to the class of qualitative methods in inverse scattering. This approach was developed for the classical Helmholtz equation and aims to recover the support of the inhomogeneity (nonlinear information) from the study of the range of the (linear) far-field operator $F$ defined in terms of far-field data (we refer the reader to the monograph \cite{CCH23} for the state of the art of this approach). The factorization method is one of the most elegant and rigorous qualitative methods for inverse scattering problems in the classical setting; for inhomogeneous media it was introduced by Kirsch in 1998 in \cite{K98} (see also \cite{kirsch2007factorization}). This non-iterative method provides a mathematically precise characterization of the support of an unknown inhomogeneity from the range of $(F^*F)^{1/4}$, where $F$ is the far-field operator. Its strength is twofold: it leads to an efficient reconstruction algorithm based on solving a linear ill-posed equation, and it provides a theoretical uniqueness result for the support of the inhomogeneity from far-field data $u^\infty(\hat x; \theta)$, $\hat x,\theta \in {\mathbb S}^{d-1}$, at a fixed frequency satisfying the required assumptions. More specifically, the method works if the contrast $n-1$ is uniformly positive or negative in $\overline{D}$.

The goal of this paper is to develop the factorization method for the fractional scattering problem (\ref{inhomogeneous_fractional_Helmholtz})-(\ref{inci}) and (\ref{SRC}). This paper is  the first work to extend qualitative methods to fractional scattering problems. This task  is enabled by \cite{zilberberg2026limiting}, where the authors compute the fundamental solution of the $s$-Helmholtz equation, establish well-posedness of the direct problem via the Lippmann–Schwinger equation, and analyze the asymptotic behavior of the scattered field at infinity. This, in turn, allows  the definition of the far-field pattern and the far-field operator. The analysis of the direct problem excludes a discrete (possibly empty) set of frequencies known as scattering poles \cite[Theorem 4.2]{zilberberg2026limiting}. In addition, the study of the injectivity of the far-field operator introduces a new set of transmission eigenvalues which, under the above assumptions on $n$, form at most a discrete set. Our results require that the fixed frequency $k$ is not one of these real eigenvalues. To validate our analysis, we provide numerical examples for both the direct and inverse problems. Note that the factorization method does not rely on solving the forward problem. However, for the reconstruction algorithm we compute simulated far-field data by solving the forward problem, which is obtained via the Lippmann–Schwinger volume integral equation developed in \cite{zilberberg2026limiting}. The explicit computation of the fundamental solution for all $0<s<1$ and $d=1,2,3$ provides the key tool to carry out these computations.

\noindent
The paper is organized as follows. Section 1, is devoted to the solution of the inverse problem under consideration. More precisely, we develop the factorization method for the fractional Helmholtz equation for $0<s<1$ in dimension $d=1,2,3$ for the  determination of  the support of the inhomogeneity, where, as part of the analysis, a new set of transmission eigenvalues is introduced. We prove that this set is discrete (possibly empty) and can only accumulate at $+\infty$. In Section 2, we present a numerical implementation for solving the forward problem and provide reconstruction examples for the support of the inhomogeneity using the factorization method for the fractional Helmholtz equation, in the case $d=2$. For readers' convenience and to avoid technicalities, we defer to the appendix a detailed discussion on the asymptotic behavior and estimates of the fundamental solution for the $s$-Helmholtz equation.

\section{Inverse  \texorpdfstring{$s$}{alpha}-Helmholtz Inhomogeneous Scattering Problem}

 As already discussed in the introduction, the fractional scattering problem for the inhomogeneous media with refractive index $n\in L^\infty({\mathbb R}^d)$ and compactly supported contrast  $D:=\supp(n-1)$, due to an incident field  $u^{inc}$ solution of the homogeneous $s$-Helmholtz equation (\ref{inci}) reads: Find the total field $u^{tot} = u + u^{inc}$ satisfying 
 $$
(\FLs -k^{2s}n ) u^{tot} = 0 \quad \mbox{ in } \R^d
$$
where the scattered field $u$ is the unique solution in $H^{2s,-\delta}(\R^d)$ for some $\delta \in (\frac{1}{2},1)$ of  
\begin{align}\label{scattering_pb}
    \begin{cases}
        (\FLs - k^{2s}n) u = k^{2s}(n-1) u^{inc} \qquad \mbox{in} \, {\mathbb R}^d\\
        u \mbox{ satisfies } (\ref{SRC}).
    \end{cases}
\end{align}
The equation for the scattered field $u$ can be re-written as
$$ (\FLs - k^{2s}) u = k^{2s}(n-1) u^{tot}  \qquad \mbox{in} \, {\mathbb R}^d.$$
 Let us denote the radiating fundamental solution of the fractional Helmholtz equation by  $\Phi_{s,k}$ and the Helmholtz fundamental solution by $\Phi_{helm,k}$. In \cite{zilberberg2026limiting}, it is proved that $\Phi_{s,k}$ in dimension $d=1,2,3$ for $s\in (0,1)$ is given by $\frac{k^{2-2s}}{s}\Phi_{helm,k} + \Phi^\Delta_{s,k}$, where $\Phi^\Delta_{s,k}$ is a term which decays faster at infinity (see the appendix for details, or \cite{zilberberg2026limiting}). Then
$$(\FLs -k^{2s}) \Phi_{s,k}(|x|)=\delta(0) \qquad \mbox{in} \, {\mathbb R}^d\, \qquad \mbox{together with}\; \ref{SRC}$$ 
and $u \in H^{2s,-\delta}(\R^d)$ is given by
\begin{align}\label{u_integral_rep}
 u(x)=k^{2s} \int_{D}  \Phi_{s,k}(|x-y|) (n(y)-1) u^{tot}(y) dy.
\end{align}
The above formula only requires to know $u^{tot}$ in $D$, which is a solution of the Lippmann-Schwinger equation
\begin{equation}\label{LSh}
(I-T_{s,k})u^{tot}=u^{inc}|_D
\end{equation}
with $T_{s,k}:L^2(D)\to L^2(D)$ defined by 
$$T_{s,k}: \phi \mapsto  k^{2s}\int_{D} \Phi_{s,k}(|x-y|) (n(y)-1) \phi(y) dy \qquad \mbox{for }x\in D.$$
In \cite[Section 4.1]{zilberberg2026limiting} it is shown that $(I-T_{s,k})^{-1}$ exists for all  real $k>0$ outside a  set  $\Lambda \subset \R^+$ (possibly empty) with infinity as its only possible accumulation point, referred to as real scattering poles,  and hence for such $k>0$ the $s$-Helmholtz scattering problem has a unique solution  $u\in H^{2s,-\delta}(\R^d)$. From now on we assume
\begin{assumption}\label{assumption_L}
The frequency $k>0$ is not in the set $\Lambda$, hence it is not a scattering pole.
\end{assumption}

Assume now that the media is probed by incident plane waves $u^{inc}(x):=e^{i k \theta \cdot x}$ with incident direction $\theta \in \mathbb{S}^{d-1}$ (such fields are solutions to the homogeneous $s$-Helmholtz equation, as shown in \cite{cheng2023equivalence}). Let  $u(\cdot, \theta)$ be the corresponding scattered field.  Since $u$ is expressed as the convolution $u(x;\theta):= k^{2s}\left(\Phi_{s,k}*(n-1)u^{tot}(\cdot;\theta) \right)(x)$, using the asymptotic expansion of the fundamental solution (see Appendix, Lemma \ref{lemma_Phi_delta} item (i) and formula (\ref{asymotic_fund_sol_Helm})) we have
$$
u(x; \theta) = c_d k^{\frac{d-3}{2}}\frac{e^{ik|x|}}{(\pi|x|)^{\frac{d-1}{2}}}u^\infty(\hat x; \theta) + O(|x|^{-\frac{d}{2}}) \quad \mbox{ as } |x| \to \infty
$$
where the function $u^\infty(\hat x; \theta)$  defined on ${\mathbb S}^{d-1}$ for fixed $\theta$ is called the far field pattern of the scattered field, and is given by
\begin{equation} \label{far-field}
u^\infty (\hat x; \theta) := \frac{k^{2-2s}}{s}\int_D e^{-ik \hat x \cdot y} k^{2s}(n(y)-1)u^{tot}(y; \theta) dy.
\end{equation}
\begin{remark}
More generally if $u$ is the unique radiating solution in $H^{2s, -\delta}(\R^d)$ of 
\begin{align} \label{s-Helm}
(\FLs-k^{2s}) u = \psi
\end{align}
with $\psi \in L^2(D)$, then the far field pattern of $u$ is given by
$$
u^\infty (\hat x) := \frac{k^{2-2s}}{s}\int_D e^{-ik \hat x \cdot y} \psi(y) dy
$$
and verifies
$$
u(x) = c_d k^{\frac{d-3}{2}}\frac{e^{ik|x|}}{(\pi|x|)^{\frac{d-1}{2}}}u^\infty(\hat x) + O(|x|^{-\frac{d}{2}}) \quad \mbox{ as } |x| \to \infty.
$$
\end{remark}
The {\bf{ inverse problem}} of our concern is to determine $D$ from a knowledge of the far-field data $u^\infty(\hat x; \theta)$ for all  $\hat x,\theta \in {\mathbb S}^{d-1}$. This scattering data defines the {\em{far field operator}} $F:L^2({\mathbb S}^{d-1})\to L^2({\mathbb S}^{d-1})$ by
\begin{equation}
(Fg)(\hat x)=\int_{\mathbb S^{d-1}}g(\theta)u^\infty(\hat x; \theta)d\theta.
\end{equation}
By linearity with respect to the incident wave, $Fg $ corresponds to the far field pattern $u^\infty_g$ of the scattered field $u_g$, where $u_g$ is the solution to the scattering problem (\ref{scattering_pb}) with incident field the superposition of plane waves $u^{inc}_g(x) = \int_{\mathbb S ^{d-1}} g(\theta) e^{ik x \cdot \theta} d\theta$, referred to in the literature as Herglotz wave function. This operator is the main imaging object of  the factorization method which we develop next.
\subsection{Factorization Method for the Characterization of the Support  \texorpdfstring{$D$}{alpha}}
In what follows, the results hold for contrasts $n-1$ that are uniformly one sign in $D$. That is $1-n>c$ or $n-1 >c$ for some $c>0$, and $n>0$. For simplicity of presentation, we will only consider the case $n-1>c$. 
\begin{assumption}\label{assumption_n}
    The refractive index $n \in L^\infty(\R^d)$ is real valued, satisfies that $\supp (n-1) =: D$ is compact, and there exists some $c>0$ such that $n(x)-1 > c$, for all $x\in D$.
\end{assumption}

The Factorization Method is based on the following functional analysis Theorem due to  Kirsch and Grinberg, see for example \cite{kirsch2007factorization}, Theorem 1.23. 
\begin{theorem}[Kirsch, Grinberg]\label{KirschTheorem}
Let $H$ be a Hilbert space, $X$ a reflexive Banach space and let the compact operator $F : H \to H$ have a factorization of the form
$$
F = BAB^*
$$
with operators $B : X \to H$ and $A : X^* \to X$ such that $Im <\phi,A \phi > \not = 0$ for all $\phi \in  \overline{ Ran(B^*)}$ with $\phi \not = 0$. Let furthermore $A$ be of the form $A = A_0 + C$ for some compact operator $C$ and some self‐adjoint operator $A_0$ which is coercive on $Ran(B^*)$ (i.e. $ \exists c>0 |\langle\phi, A_0\phi \rangle \geq c \langle \phi, \phi\rangle$ for all $\phi \in Ran(B^*)$). Finally, assume that $F$ is one‐to‐one and $I + irF$ is unitary for some $r > 0$. Then the ranges of $B$ and $(F^*F)^{1/4}$  coincide. Furthermore, the operators $(F^*F)^{-1/4}B$ and $B^{-1}(F^*F)^{1/4}$  are isomorphisms from $X$ onto $H$ and from $H$ onto $X$, respectively.
\end{theorem}

We will use the above Theorem on the far field operator $F$. First, we show that $F$ satisfies the following factorization. 
\begin{lemma}
 Assuming \ref{assumption_L} and \ref{assumption_n}, 
 \begin{align}\label{factorization_G}
F = \frac{1}{sk^2} G (I-T_{s,k})^{-1}\frac{1}{(n-1)} G^* 
\end{align}
where $G : L^2(D) \to L^2(\Sp^{d-1})$ is defined by $G f = v^\infty$, the far field pattern of the radiating solution to $(-\Delta -k^2) v= k^2 (n-1) f$, for $ f\in L^2(D)$. 
\end{lemma}
\begin{proof}
    We observe from formula (\ref{far-field}) and (\ref{LSh}) and by linearity that for all $g \in L^2(\mathbb S^{d-1})$ 
\begin{align}\label{equality_farfield}
u_g^\infty(\hat x) = \frac{1}{s} v_g^\infty(\hat x) 
\end{align}
where $v_g^\infty$ is the far field pattern of the radiating solution $v$ to the Helmholtz equation 
\begin{align*}
(-\Delta - k^2) v = k^2 (n-1) (I-T_{s,k})^{-1} \left.\int_{\mathbb S^{d-1}} g(\theta)e^{ik x \cdot \theta}ds(\theta)\right|_D.
\end{align*}
Hence $v_g^\infty = G (I-T_{s,k})^{-1}\left. \int_{\mathbb S^{d-1}} g(\theta)e^{ik  \left<\theta ,\cdot\right> }ds(\theta)\right|_D$. Let us introduce the Herglotz wave operator $H : L^2(\mathbb S^{d-1})\to L^2(D)$ defined by
\begin{align*}
 Hg(x) = \left.\int_{\Sp^{d-1}} g(\theta)e^{i k \theta \cdot x }  ds(\theta) \right|_D = u^{inc}_g|_D(x)\quad.
\end{align*}
Then we see that 
\begin{align}\label{first_factorization}
    Fg = u^\infty_g = \frac{1}{s} v^\infty_g = \frac{1}{s} G (I-T_{s,k})^{-1} Hg.
\end{align}
We now express $H$ in terms of $G^*$. Let us first note that 
\begin{align*}
    H^* f (\theta) = \int_D e^{-ik \theta \cdot x} f(x) dx = v^\infty(\theta)
\end{align*}
where $v^\infty$ is the far field pattern of $v$ solution of 
$$
(-\Delta - k^2) v = f.
$$
Hence $H^* f= G \frac{f}{k^2(n-1)}$ and $H = \frac{1}{k^2 (n-1)}G^*$. Combining this identity with (\ref{first_factorization}), we have proved the desired factorization of the far field operator.
\end{proof}
The recovery of the support of the inhomogeneity relies on the fact that the range of $G$ characterizes the support of $n-1$, as described in the following classical lemma. Let us write the point source wave $\Phi( x,z) :=\Phi_{helm,k}(|x-z|) $ and its far field pattern $\Phi^\infty(\cdot, z) = e^{-ik \langle \cdot, z \rangle}$. 
\begin{lemma}\label{RanG}
     $\Phi^\infty(\cdot, z) \in Ran( G)$ if and only if $z \in D$.
\end{lemma}
\begin{proof} The proof follows the lines of Lemma 11.13 in \cite{colton1998inverse}, we write it here for convenience. Let $z\in D$ and $\epsilon$ small enough such that $B_\epsilon(z) \subset D$. Define $\rho$ to be a radial smooth function such that $\rho = 0$ in $B_{\epsilon/2}(z)$ and $\rho = 1$ in $\R^d \setminus B_\epsilon(z)$. Set $v(x) = \rho(x) \Phi(x,z) \in C^\infty(\R^d)$, then the far field pattern of $v$ is $\Phi^\infty(\cdot, z)$ and it is equal to $G f$, where $f := \left( \frac{1}{k^2 (n-1)} (-\Delta -k^2)v\right)\in L^2(D)$.\\
Let $z\not \in D$. Suppose by contradiction that there exists $f\in L^2(D)$ such that $Gf = \Phi^\infty(\cdot, z)$. Let $v$ be the radiating solution of $(-\Delta - k^2) v= f$ in $\R^d$. Then by Rellich Lemma, $\Phi(\cdot, z) \equiv v$ in the exterior of $D \cup \{z\}$. This contradicts the regularity of $v$. 
\end{proof}
As mentioned in the introduction, we will relate through the standard Kirsch-Grinberg Theorem the ranges of $G$ and $(F^*F)^{1/4}$. However, we need to exclude a set of transmission eigenvalues, that we first need to define in the context of the fractional Helmholtz equation.
\begin{definition}\label{TransmissionEigenvalue}
  We say that $k \in \R^+ \setminus \Lambda$ is a transmission eigenvalue for the inhomogeneous $s$-Helmholtz equation with index $n$ if $k$ satisfies Assumption \ref{assumption_L} and if there exists a non trivial $v \in H^2_0(D)$ and a $w \in L^2(D)$ such that 
   \begin{align}\label{system_te}
   \left\{ \begin{array}{rll}
       (-\Delta - k^2) v &= k^2 (n-1) (I-T_{s,k})^{-1}w \quad &\mbox{ in } D\\
    (-\Delta -k^2) w &= 0 \quad &\mbox{ in } D 
   \end{array} \right.\quad .
    \end{align}
\end{definition}

We can now show that the far field operator satisfies the assumptions of Theorem \ref{KirschTheorem}, provided $k$ is not a transmission eigenvalue as described above. 
\begin{lemma}\label{assumptionF}
    If $k$ is not a transmission eigenvalue in the sense of Definition \ref{TransmissionEigenvalue}, the far field operator $F$ is compact, injective, and $I+\frac{2i |c_d|^2 k^{d-2}}{\pi^{d-1}} F$ is unitary. 
\end{lemma}
\begin{proof}
    $F$ is compact by the compactness of the embedding $H^1(\Sp^{d-1}) \hookrightarrow L^2(\Sp^{d-1})$. Let $g\in L^2(\Sp^{d-1})$ be such that $Fg = 0$. Let $u$ be the scattered field solution of (\ref{scattering_pb}) with $u^{inc}|_D = Hg$. Let $w:= Hg$ and $v := \Phi_{helm,k}*k^{2}(n-1)(I-T_{s,k})^{-1}w $. Then $(v,w)\in H^2_0(D) \times L^2(D)$ and  
    \begin{align*}
   \left\{ \begin{array}{rll}
       (-\Delta - k^2) v &= k^2 (n-1) (I-T_{s,k})^{-1}w \quad &\mbox{ in } D\\
    (-\Delta -k^2) w &= 0 \quad &\mbox{ in } D 
   \end{array} \right. \quad.
    \end{align*}
    Since $k$ is not a transmission eigenvalue, $v = w = 0$. In particular $g = 0$. \\
    For a function $g\in L^2(\Sp^{d-1})$, we introduce $u_g$ the radiating solution of the scattering problem (\ref{scattering_pb}) with incident field $Hg$. Let us denote by $v_g$ its Helmholtz part and by $w_g$ the remainder part such that $u_g = v_g + w_g$ and
    \begin{align}
         w_g &= \Phi_{s,k}^\Delta*k^{2s}(n-1) u^{tot}_g \label{eq_wgs}\\
         v_g &= s^{-1}\Phi_{helm,k}*k^{2}(n-1) u^{tot}_g \label{eq_vgs}
    \end{align}
   where $u^{tot}_g = u_g + Hg$ denotes the total field. Using (\ref{eq_vgs}), we see that the total Helmholtz field $v^{tot}_g = v_g + Hg $ satisfies the equation 
\begin{align}\label{eq_tot_helm}
   (-\Delta - k^2 - s^{-1} k^2(n-1)) v^{tot}_g = s^{-1}k^2(n-1) w_g .
    \end{align}
    Let $R$ be large enough such that $D \subset B_R$. Using the notations introduced above for $g,h \in L^2(\Sp^{d-1})$ and using (\ref{eq_tot_helm}) we have
\begin{align*}
        \int_{B_R} v^{tot}_g \Delta \overline{v^{tot}_h} - \overline{v^{tot}_h}\Delta v^{tot}_g &=  \int_D \overline{v^{tot}_h} s^{-1}k^2 (n-1) w_g - \int_D v^{tot}_g s^{-1}k^2 (n-1) \overline{w_h}\\
        &= \int_D \overline{u^{tot}_h} s^{-1}k^2(n-1)w_g\quad -\int_D u^{tot}_g s^{-1}k^2 (n-1) \overline{w_h} .
    \end{align*}
Since the convolution operator $\Phi_{s,k}^\Delta*$ is self-adjoint (see item (i) of Lemma (\ref{lemma_Phi_delta})) and $w_g, w_h$ are given by (\ref{eq_wgs}), we deduce from the above calculation that for all $R$ large enough such that $D \subset B_R$
\begin{align*}
       \int_{B_R} v^{tot}_g \Delta \overline{v^{tot}_h} - \overline{v^{tot}_h}\Delta v^{tot}_g&= 0\quad.
\end{align*}
The rest of the arguments follow the lines of the proof of Theorem 1.8 in \cite{kirsch2007factorization}, we write them here for convenience. After an integration by parts, we obtain 
\begin{align}\label{partialBR}
    \int_{\partial B_R} v^{tot}_g \frac{\partial \overline{v^{tot}_h}}{\partial\nu} ds -  \overline{v^{tot}_h} \frac{\partial v^{tot}_g}{\partial\nu} ds =0.
\end{align}
The integral can be split into four terms by decomposing $v^{tot}_g = v_g + Hg$ and $v^{tot}_h = v_h + Hh$. The term containing only the incident waves $Hg$ and $Hh$ vanishes by Green's second identity. It is straightforward to see that $\frac{\partial}{\partial |x|} \Phi_{helm,k}(x) \sim ik \Phi_{helm,k}(x)$ at infinity. Hence, using the asymptotic formula for $\Phi_{helm,k}$ (\ref{asymotic_fund_sol_Helm}), the scattered fields verify asymptotically 
\begin{align*}
    v_g \frac{\partial \overline{v_h}}{\partial\nu} ds -  \overline{v_h} \frac{\partial v_g}{\partial\nu} = -2i|c_d|^2 k^{d-2} \frac{v^\infty_g(\hat x) \overline{v_h^\infty}(\hat x)}{(\pi |x|)^{d-1}} + O(|x|^{-d}) .
\end{align*}
We deduce that 
\begin{align*}
    \int_{\partial B_R} v_g \frac{\partial \overline{v_g}}{\partial\nu} ds -  \overline{v_g} \frac{\partial v_g}{\partial\nu} ds  \xrightarrow[R \to \infty]{} -\frac{2i|c_d|^2 k^{d-2}}{\pi^{d-1}}\int_{\Sp^{d-1}} v_g^\infty \overline{v_h^\infty} ds = -\frac{2i|c_d|^2 k^{d-2}}{\pi^{d-1}}(Fg, Fh).
\end{align*}
We now use the representation of $v_{g}, v_h$ outside $B_R$ given by Green's representation formula (see for example Theorem 1.3 in \cite{kirsch2007factorization})
\begin{align*}
    v_{g}(x) = \int_{\partial B_R} v_{g}(y) \frac{\partial \Phi_{helm,k}(x-y)}{\partial \nu} - \frac{\partial v_{g}(y)}{\partial \nu} \Phi_{helm,k}(x-y) ds(y), \quad x\in \R^d \setminus \overline{B_R}
\end{align*}
to obtain the following formula for the far field pattern  
\begin{align*}
    v_{g}^\infty (\hat x) = \int_{\partial B_R}  v_{g}(y) \frac{\partial e^{-i \hat{x}\cdot y}}{\partial \nu} - \frac{\partial v_{g}(y)}{\partial \nu} e^{-ik \hat{x}\cdot y} ds(y).
\end{align*}
The formulas hold for $v_h$ as well. We deduce that 
\begin{align*}
    \int_{\partial B_R} Hg \frac{\partial \overline{v_h}}{\partial \nu} - \overline{v_h} \frac{\partial Hg}{\partial \nu} ds &= \int_{\Sp^{d-1}} g(\theta) \int_{\partial B_R} \frac{\partial \overline{v_{h}}(y)}{\partial \nu} e^{ik \theta\cdot y} - \overline{v_{h}}(y) \frac{\partial e^{i \theta\cdot y}}{\partial \nu} ds(y) d\theta\\
    &= - (g, Fh)
\end{align*}
and similarly 
\begin{align*}
    \int_{\partial B_R} v_g \frac{\partial \overline{Hh}}{\partial \nu} - \overline{Hh} \frac{\partial v_g}{\partial \nu} ds &= (Fg, h).
\end{align*}
Taking $R \to \infty$ in (\ref{partialBR}), we obtain 
 $$
 0 = -(g,Fh) + (Fg,h)  - \frac{2i |c_d|^2 k^{d-2}}{\pi^{d-1}} (Fg,Fh) \quad \mbox{ for all } g,h \in L^2(\Sp^{d-1}),
 $$
 that is $F-F^* - \frac{2i |c_d|^2 k^{d-2}}{\pi^{d-1}} F^*F=0.$
 We deduce that 
 $$
 \left(I+\frac{2i |c_d|^2 k^{d-2}}{\pi^{d-1}} F\right)^* \left(I+ \frac{2i |c_d|^2 k^{d-2}}{\pi^{d-1}} F\right) = I + \frac{2i |c_d|^2 k^{d-2}}{\pi^{d-1}}\left( F -F^* - \frac{2i |c_d|^2 k^{d-2}}{\pi^{d-1}} F^* F\right) = I .
 $$
 Hence we proved that $I + i r F$ is unitary with $r=\frac{2 |c_d|^2 k^{d-2}}{\pi^{d-1}} >0$. 
\end{proof}
We now verify the assumptions of Theorem \ref{KirschTheorem} on the middle operator.
\begin{lemma} \label{assumptionS}The operator $S := (I-T_{s,k})^{-1}\frac{1}{(n-1)}$ satisfies
\begin{itemize}
 \item[(i)] $S = S_0 + C$ where $C$ is compact, $\langle S_0f,f \rangle \in \R$ and $\langle S_0f,f \rangle \geq c \|f\|^2_{L^2(D)}$ for some $c>0$ and for all $f \in \overline{Ran(G^*)}$,
    \item[(ii)]  $Im ( f, Sf) \not = 0 $ for all $f \in \overline{Ran(G^*)}\setminus \{0\}$.
\end{itemize}
   \end{lemma}
\begin{proof}
    \begin{itemize}
        \item[(i)] Noticing that $(I-T_{s,k})^{-1} = I + T_{s,k}(I-T_{s,k})^{-1}$, the decomposition is clear from the definition of $S$, we take
        \begin{align*}
             S_0 f:= \frac{f}{(n-1)} \quad \mbox{and} \quad C f := T_{s,k}(I-T_{s,k})^{-1}\frac{f}{(n-1)}.
        \end{align*}
$C$ is compact by the compact embedding $H^{2s}(D) \hookrightarrow L^2(D)$. By Assumption \ref{assumption_n} on $n$, $S_0$ is bounded and
$$
\langle S_0 f, f\rangle \geq \frac{1}{ \|n-1\|_\infty} \|f\|^2_{L^2(D)}. 
$$
\item[(ii)] Let us assume that $f \in \overline{Ran(G^*)} = ker(G)^\perp$ and $Im(f, Sf) = - Im (Sf, f) = 0$. We will show that $f = 0$. Since for all $\phi \in C_c^\infty(D)$, $\frac{1}{k^2(n-1)} (-\Delta-k^2)\phi \in \ker G$ , we have
        \begin{align*}
        \int_D f \overline{g} =0 \quad \forall g \in \ker G \quad  \implies \quad \int_D f \frac{1}{(n-1)k^2}(-\Delta-k^2)\overline{\phi} =0 \quad \forall \phi \in C_c^\infty(D).
    \end{align*}
    Hence $w_0 :=\frac{f}{s k ^{2s}(n-1)} \in L^2(D)$ is a solution of 
    $$
    (-\Delta -k^2) w = 0 \mbox{ in } D
    $$
    in the sense of distributions. We chose the constant $\frac{1}{sk^{2s}}$ for future convenience. Using $(I-T_{s,k})^{-1} = I + T_{s,k}(I-T_{s,k})^{-1}$ and formula (\ref{u_integral_rep})-(\ref{LSh}), we have \begin{align*}
       & I_m((I-T_{s,k})^{-1}\frac{f}{(n-1)}, f) = 0\\
        \iff & I_m(T_{s,k}(I-T_{s,k})^{-1}\frac{f}{k^{2s}(n-1)}, f) = 0\\
        \iff &I_m(u, f) = 0
    \end{align*} 
    where $u$ is the radiating solution of $(\FLs -k^{2s}n) u = f$. Let \begin{align}
        v &:= \Phi_{helm,k} * \frac{k^{2-2s}}{s} (f+ k^{2s}(n-1)u) \label{vs}\\
        w &:= \Phi_{s,k}^\Delta *(f+ k^{2s} (n-1)u),\label{ws}
    \end{align}
    then $u = v + w$. Using that $(-\Delta -k^2) v = \frac{k^{2-2s}}{s} f + \frac{k^2}{s}(n-1) u $, we compute 
    \begin{align*}
        \int_D u \overline{f} dx&= \int_D v \overline{f} + \int_D w \overline{f}\\
        &=\int_D v \frac{s}{k^{2-2s}}\left((-\Delta -k^2) \overline{v} - \frac{k^2}{s}(n-1) \overline{u}\right) + \int_D \Phi_{s,k}^\Delta *(k^{2s}(n-1) u) \overline{f} + \int_D \Phi_{s,k}^\Delta*f \overline{f} .
    \end{align*}
We see that since $\Phi_{s,k}^\Delta$ is real valued (see item (i) of Lemma \ref{lemma_Phi_delta}), the last term will not have a contribution to the imaginary part. After integrating by part in the first term and taking the imaginary part, only the boundary term will remain. Next, we use that $\Phi_{s,k}^\Delta$ is self-adjoint and we obtain   
\begin{align}\label{step1}
  Im  \int_D u \overline{f} =& -\frac{s}{k^{2-2s}} Im\int_{\partial D} v \frac{\partial \overline{v}}{\partial \nu} ds - k^{2s} Im \int_D (n-1)v \overline{u} + k^{2s} Im \int_D (n-1)u \Phi_{s,k}^\Delta *\overline{f}.
  \end{align}
Using (\ref{ws}), the last term can be rewritten
  \begin{align*}
k^{2s} Im \int_D (n-1)u (\overline{w} - \Phi_{s,k}^\Delta * (k^{2s}(n-1) \overline{u}) = k^{2s} Im \int_D (n-1)u \overline{w} 
\end{align*}
hence, using that $u = v + w$, the last two terms of  (\ref{step1}) cancel and 
\begin{align*}
    0 = Im \int_D u \overline{f} = -\frac{s}{k^{2-2s}} Im\int_{\partial D} v \frac{\partial \overline{v}}{\partial \nu} ds. 
\end{align*}

Using Theorem 2.13 in \cite{colton1998inverse}, we conclude that $v \equiv 0$ in $\R^d \setminus \overline D$ and $v, w_0 \in H^2_0(D)\times L^2(D)$ are solutions to 
\begin{align*}
    (-\Delta -k^2) v = k^{2}(n-1) (I-T_{s,k})^{-1} w_0 \quad &\mbox{ in } D\\
    (-\Delta -k^2) w_0 = 0 \quad &\mbox{ in } D.
\end{align*}
Since $k$ is not a transmission eigenvalue, $v=w_0=0$, hence in particular $f = 0$.
\end{itemize}
\end{proof}
Combining Theorem \ref{KirschTheorem}, the factorization (\ref{factorization_G}) and Lemmas \ref{RanG}, \ref{assumptionF}, \ref{assumptionS}, we proved the following Theorem which is the main result of this paper.
\begin{theorem}\label{FM_main}
Let $0<s<1$, and $d = 1,2,3$. Assume \ref{assumption_n} and that $k$ is not a transmission eigenvalue in the sense of Definition \ref{TransmissionEigenvalue}. Then 
    $\Phi^\infty(\cdot, z) \in Ran( (F^* F)^{1/4})$ if and only if $z \in D$ .
\end{theorem}
\noindent An immediate corollary of Theorem \ref{FM_main} is the following uniqueness theorem for the support of the inhomogeneity from far field measurements at any frequency $k \in \R^+$ (except a discrete set of exceptional values, see Theorem \ref{discrete_transmissionEigenvalues} for the discreteness of the transmission eigenvalues). 
\begin{theorem}
    Let $0<s<1$, and $d = 1,2,3$. Assume $n_1$ and $n_2$ satisfy \ref{assumption_n} and that $k$ not a transmission eigenvalue for $n_1$ and $n_2$ in the sense of Definition \ref{TransmissionEigenvalue}. Then if the far field patterns $u^\infty_j$ corresponding to $n_j$  coincide, that is, $u_1^\infty(\hat x; \theta) = u^\infty_2(\hat x; \theta)$ for all $\hat x, \theta \in \mathbb{S}^{d-1}$, the supports $D_1$ and $D_2$ of $n_1-1$ and $n_2-1$ coincide.
\end{theorem}

\subsection{Discreteness of the Transmission Eigenvalues}
In this section, we come back to the notion of transmission eigenvalue introduced in Definition \ref{TransmissionEigenvalue}. We show that the set of such values is discrete, and can only accumulate at $+\infty$, hence it is not an unreasonable assumption for Theorem \ref{FM_main}.
\begin{theorem}\label{discrete_transmissionEigenvalues}
    The set of transmission eigenvalues $k \in \R^+\setminus \Lambda$ in the sense of Definition \ref{TransmissionEigenvalue} is discrete, possibly empty, and can accumulate only at $+\infty$. 
\end{theorem}
\begin{proof}
Let $(v,w) \in H_0^2(D)\times L^2(D)$ be solutions of (\ref{system_te}). Dividing the first equation by $k^2(n-1)$, applying $(I-T_{s,k})$ on both sides and multiplying by $(-\Delta -k^2) \phi $ for $\phi\in C_c^\infty(D)$, we obtain
$$
\int_D (I-T_{s,k})\left(\frac{1}{(n-1)} (-\Delta -k^2) v\right) \left( (-\Delta - k^2) \phi \right) = 0 \quad \forall \phi \in C_c^\infty(D).
$$
Expanding the terms and using that $T_{s,k} f= \frac{k^2}{s} \Phi_{helm,k}*(n-1) f + \Phi^\Delta_{s,k}*k^{2s}(n-1)f$, we obtain 
\begin{align}\label{var_te}
    \int_D  \Delta v \Delta \phi \frac{dx}{(n-1)} +k^2 \int_D (\Delta v \phi + v \Delta \phi)  \frac{dx}{(n-1)} + k^4 \int_D v \phi \frac{dx}{(n-1)} &\\+ \frac{k^2}{s}\int_D v (\Delta +k^2) \phi dx -k^{2s}\int_D \left(\Phi^\Delta_{s,k}*(\Delta +k^2) v \right)(\Delta +k^2)\phi dx &= 0\quad \forall \phi \in C_c^\infty(D).\nonumber
\end{align}
We can define the bounded linear forms on $H^2_0(D) \times H^2_0(D)$ 
\begin{align*}
    a_0(v,\phi) = & \int_D  \Delta v \Delta \phi \frac{dx}{(n-1)} \\
    a_1(k)(v,\phi) =& k^2\int_D (\Delta v \phi + v \Delta \phi)  \frac{dx}{(n-1)} +\frac{k^2}{s}\int_D v \Delta \phi dx + k^4\int_D v \phi \left( \frac{1}{(n-1)} + \frac{1}{s} \right) dx\\
    &- k^{2s}\int_D \left(\Phi^\Delta_{s,k}*(\Delta +k^2) v \right)(\Delta +k^2)\phi dx .
\end{align*}
By Riesz representation Theorem, there exists bounded linear operators $A_0, A_1(k) : H^2_0(D) \to H^2_0(D)$ such that $a_0(v,\phi) = (A_0 v, \phi)_{H^2(D)} $ and $a_1(k)(v,\phi) = (A_1(k) v , \phi)_{H^2(D)}$ for all $v,\phi \in H^2_0(D)$. We show that $A_1(k)$ is compact. Let $\phi_n$ be a sequence converging weakly to $0$ in $H^2_0(D)$. We have the bound 
\begin{align}\label{boundA1}
    \| A_1(k) \phi_n \|_{H^2_0(D)}^2 =& a_1(k) (\phi_n, A_1(k)\phi_n) \nonumber\\
   \leq &C(n,s)\left(( k^2+k^4)(\|\phi_n\|_{H^2(D)} \| A_1(k) \phi_n\|_{L^2(D)}+ \|\phi_n\|_{L^2(D)} \| A_1(k) \phi_n\|_{H^2(D)})\right.\nonumber\\
   &+\left. k^{2s}\| \Phi_{s,k}^\Delta*(\Delta+k^2)\phi_n \|_{L^2(D)} \|A_1(k)\phi_n\|_{H^2(D)}\right) .
\end{align}
Since $\phi_n$ converges weakly to $0$ in $H^2(D)$, $\|\phi_n\|_{H^2(D)}$ and $\|A_1(k) \phi_n\|_{H^2(D)}$ are bounded. By the compact embedding $H^2_0(D) \hookrightarrow L^2(D)$, the sequence $\phi_n$ converges strongly to $0$ in $L^2(D)$, as well as $A_1(k) \phi_n$. Finally, the convolution operator $\Phi_{s,k}^\Delta*$ is compact (see Lemma \ref{lemma_Phi_delta}, item (ii)), hence $\Phi_{s,k}*(\Delta +k^2)\phi_n$ also converges to $0$ in $L^2(D)$. This proves that $A_1(k)$ is compact. The variational equation (\ref{var_te}) is equivalent to 
$$
A_0 v +  A_1(k) v = 0 . 
$$
Since $a_0$ is coercive, the operator $A_0$ is invertible, and we have that $k$ is a transmission eigenvalue if and only if 
$$
(I +   A_0^{-1} A_1(k) )v  = 0
$$
has a non-trivial solution $v\in H^2_0(D)$. 
The operator $A_0^{-1}A_1(k)$ is compact (composition of a compact operator and bounded operator) and depends analytically on $k$. Furthermore, we can estimate the norm of $A_1(k)$ using (\ref{boundA1}) and item (ii) of Lemma \ref{lemma_Phi_delta} to obtain 
\begin{align*}
    \| A_1(k) \phi\|_{H_0^2(D)\to H_0^2(D)} &\leq C(n,s) (k^2 + k^4 + k^{2s} \| \Phi_{s,k}^\Delta \|_{L^1(D)})\\
    &\leq C(n,s,D) (k^s+k^4).
\end{align*}
We deduce that for $k$ small enough, $I - A_0^{-1} A_1(k)$ is invertible. Hence by analytic Fredholm theory (see for example Theorem 8.26 in \cite{colton1998inverse}), $I - A_0^{-1}A_1(k)$ is invertible for all $k\in \R^+$ except for a discrete set, possibly empty, that can accumulate only at infinity. 
\end{proof}

\section{Numerical Results}

\subsection{The Direct Problem}

To generate the far field operator numerically, we need to solve the Lippmann-Schwinger equation and, in particular, we need to implement the convolution with the fundamental solution. In the paper \cite{zilberberg2026limiting}, we proved that for a function $f\in L^{2,\delta}(\R^d)$, $1/2<\delta<1 $, the solution of (\ref{s-Helm}) is obtained by convolving with the fundamental solution $\Phi_{s,k}$ 
\begin{align} \label{convolution}
    u(x) = \int_{\R^d} \Phi_{s,k}(x-y) f(y) dy. 
\end{align}
Note that for functions $f$ that are not compactly supported, an error comes from truncating the above integral. We define the region of interest to be the $d$-dimensional cube $[-x_{max}, x_{max}]^d$ centered at $0$ and discretized into $N_x^d$ smaller cubes of size $h$ centered at points $x_i$, for $i = 1,...,N_x^d$. We then classically approximate the integrand by a piecewise constant function on each cube, with value the evaluation of the integrand at the center of that cube. This gives rise to the following numerical approximation of (\ref{convolution})
\begin{align*}
    u_{approx}(x_i) = \sum_{j=1, j\not = i }^{N_x^d} \Phi^{approx}_{s,k}(x_i-x_j) f(x_j) h^d + f(x_i)  I\Phi^{approx}\quad .
\end{align*}
The square centered at $x_i$ is treated separately by the term $I\Phi^{approx}$, as the fundamental solution is singular at zero.  \\
\indent In what follows, we detail a method to implement $\Phi_{s,k}^{approx}$ and $I\Phi^{approx}$ in dimension $d=2$. The two dimensional case contains more challenging terms than dimensions 1 and 3, but dimensions 1 and 3 would use similar techniques. Recall that in dimension 2, the fundamental solution is given by a multiple of the Hankel function $H^{(1)}_0(k|x|)$ and terms from formula (\ref{Phi_sk_delta2}). The Hankel function can be found in the Matlab library. The integral term in (\ref{Phi_sk_delta2}) is an infinite oscillatory integral that requires a truncation. The integrand at infinity is a $O\left(\frac{1}{r^{2s(m+1) + 1/2}}\right)$, hence if the integral is truncated at $C$, the error will be 
\begin{align*}
    O\left(\int_C^\infty \frac{dr}{r^{2s(m+1)+1/2}}\right) = O\left(\frac{1}{C^{2s(m+1)-1/2}}\right).
\end{align*}
We choose $C$ such that the error is a $O(h^2)$, that is 
$C = h^\frac{-2}{2s(m+1)-1/2}$. Moreover, to capture the oscillatory nature of the integral and the cancellations it ensures, we discretize it with $10$ points per "worst" period. The period of an oscillation is roughly $\frac{2\pi}{|x|}$, so the smallest period will be for the maximal computed value of $|x|$, which is $ 2\sqrt{2} x_{max}$. The number of periods in $[0,C]$ is
$$
N_{periods} = \frac{C2\sqrt{2}x_{max}}{2\pi }.
$$
Taking $10$ points per period, we obtain that the number of points used to discretize the integral will be 
\begin{align*}
    N := N_{periods} \times N_{points/period} = 10\frac{C2\sqrt{2}x_{max}}{2\pi }.
\end{align*}
We then use the trapeze method built-in matlab using the above parameters. When $s<1/2$, the additional terms are straightforward to compute, except for the Struve function $K_0(k|x|)$, for which we use the code developed in \url{https://www.mathworks.com/matlabcentral/fileexchange/37302-struve-functions}. 

Next, to compute $I\Phi^{approx}$, we use the fact that the fundamental solution is radial and approximate the integral over the square by an integral over the ball of same center and radius $h/2$. We then take a computable equivalent near zero of the fundamental solution.  
\begin{itemize}
    \item For the Helmholtz part $\Phi_{helm,k}$, the asymptotic of the Hankel function as $z \to 0$ is  
    \begin{align*}
        H^{(1)}_0(z) \sim \frac{2i}{\pi} \ln(z) .
    \end{align*}
    Hence we have  
    \begin{align}\label{approx1}
       \frac{i}{4}s^{-1}k^{2-2s} \int_{B(x_i, h/2)} H^{(1)}_0(k|x-x_i|) dx 
       &\approx   \frac{i}{4}s^{-1}k^{2-2s} 2\pi \int_0^{h/2} \frac{2i}{\pi} \ln(r k) r dr\nonumber\\
       &= -s^{-1}k^{2-2s}  \left(\frac{h^2}{8} \ln(kh/2)  - \frac{h^2}{16}\right)
    \end{align}
   
    \item For the integral term, we need to approximate
    \begin{align*}
I:=          \frac{1}{2\pi} \int_{B(x_i,h/2)} \int_0^\infty J_0(|x-x_i|r) F(r,k) r dr dx = \int_0^{h/2} \int_0^\infty ryJ_0(ry)F(r,k) drdy.
    \end{align*}
    We exchange the order of integration and use the relation 
    \begin{align*}
        \frac{d}{dx}\left(x J_1(x)\right) = xJ_0(x)
    \end{align*}
    to obtain 
    \begin{align*}
    I& =  \int_0^\infty \int_0^{h/2}\frac{d}{dy}(yJ_1(ry)) F(r,k) dydr= \int_0^\infty \frac{h}{2}J_1\left( r \frac{h}{2}\right) F(r,k) dr= \int_0^\infty J_1(r) F\left(\frac{2r}{h},k\right) dr.
    \end{align*}
    To further simplify this term, we can use the asymptotic of $F(\frac{2r}{h},k)$ as $h \to 0$ 
    \begin{align*}
        F\left(\frac{2r}{h},k\right)  \sim \frac{h^{2s(m+1)} k^{2sm}}{(2r)^{2s(m+1)}} 
    \end{align*}
    to get that 
    \begin{align}\label{mu}
        I \approx
   (h/2)^{2s(m+1)} k^{2sm}\int_0^\infty  r^{-2s(m+1)}J_1(r) dr.
    \end{align}
The formula 10.22.43 in \cite{DLMF} given below explicitly computes the integral of Bessel functions against power functions  
    \begin{align*}
        \int_0^\infty t^\mu J_\nu (t) dt = \frac{2^\mu \Gamma(\nu/2 +  \mu/2 + 1/2)}{\Gamma(\nu/2-\mu/2 + 1/2)} \quad \mbox{ for } Re(\mu+\nu) >-1, Re(\mu) <1/2.
    \end{align*}
    We apply the formula for $\nu = 1$ and $\mu = -2s(m+1)$ (they satisfy $Re(\mu+\nu) >-1$ and $Re(\mu) < \frac{1}{2}$). Therefore 
    \begin{align*}
        I_2 \approx  (h/2)^{2s(m+1)}k^{2sm}\frac{2^{-2s(m+1)}\Gamma(1- s(m+1))}{\Gamma(1+ s(m+1))}.
    \end{align*}
    \item When $s\leq \frac{1}{2}$, the fundamental solution $\Phi_{s,k}$ contains a sum of singular kernels $|x|^{-2+2s(j+1)}$. We approximate their contribution to $I\Phi^{approx}$ by 
    \begin{align}\label{kernel_approx}
    \int_{B(x_i,h/2)}\frac{c_{s,j}}{|x-x_i|^{2-2s(j+1)}} dx = \frac{\pi c_{s,j} }{s(j+1)} \left(\frac{h}{2}\right)^{2s(j+1)}.
    \end{align}
    \item Finally when $s\leq \frac{1}{2}$ and $\frac{1}{2s}\in \mathbb{N}$, we need to account for the contribution of the Struve function. Since when $z \to 0$, $K_0(z)$ is equivalent to $\frac{2}{\pi} \ln(\frac{z}{2})$, we have 
    \begin{align}\label{K0_approx}
        \int_{B(x_i,h/2)} -\frac{k^{2-2s}}{4} K_0(k|x|) dx &\sim -k^{2-2s} \int_0^{h/2} r\ln\left(\frac{r k}{2}\right) dr = - k^{2-2s}\left(\frac{h^2}{8} \ln\left(\frac{h k}{2}\right) - \frac{h^2}{16}\right).
    \end{align}
\end{itemize}
Combining (\ref{approx1}), (\ref{mu}), (\ref{kernel_approx}) and (\ref{K0_approx}), we obtain a construction of $I\Phi^{approx}$ in dimension $d=2$. \\

\indent To test the numerical method, we use the computations done in \cite{sheng2020fast}, proposition 4.2 and 4.3, where the authors obtained explicit expressions of the fractional laplacians of a Gaussian and an algebraically decaying family of functions :
\begin{align*}
    \FLs (e^{-|x|^2})  &= \frac{2^{2s}\Gamma(s+d/2)}{\Gamma(d/2)} \hgeom{1}{1}(s+\frac{d}{2}, \frac{d}{2}, -|x|^2)\\
    \FLs \left(\frac{1}{(1+|x|^2)^\alpha} \right) &= \frac{2^{2s} \Gamma(s+\alpha)\Gamma(s+d/2)}{\Gamma(\alpha)\Gamma(d/2)} \hgeom{2}{1}(s+\alpha, s+\frac{d}{2}, \frac{d}{2}, -|x|^2)
\end{align*}
where $\alpha > 0$. 
Hence we will consider the right hand sides 
\begin{align*}
    f_{exp}(x) &:=  \FLs (e^{-|x|^2}) - k^{2s}e^{-|x|^2}  \\
    &= \frac{2^{2s}\Gamma(s+d/2)}{\Gamma(d/2)} \hgeom{1}{1}\left(s+\frac{d}{2}, \frac{d}{2}, -|x|^2\right) -k^{2s}e^{-|x|^2} \\
    f_{\alpha}(x) &:=  \FLs \left(\frac{1}{(1+|x|^2)^\alpha} \right) -\frac{k^{2s}}{(1+|x|^2)^\alpha}\\&= \frac{2^{2s} \Gamma(s+\alpha)\Gamma(s+d/2)}{\Gamma(\alpha)\Gamma(d/2)} \hgeom{2}{1}\left(s+\alpha, s+\frac{d}{2}, \frac{d}{2}, -|x|^2\right) - \frac{k^{2s}}{(1+|x|^2)^\alpha}
\end{align*}
with $\alpha > \frac{d}{2}$, which give rise to the radiating solutions $u_{exp} := e^{-|x|^2}$ and $u_{\alpha}:= \frac{1}{(1+|x|^2)^\alpha}$ respectively, such that 
\begin{align*}
    (\FLs - k^{2s}) u_{\exp} &= f_{\exp}\\
    (\FLs - k^{2s}) u_{\alpha} &= f_\alpha\quad.
\end{align*}

We must verify that $f_{\exp}$ and $f_{\alpha}$ are in $L^{2,\delta}(\R^d)$ for some $1/2<\delta <1$. The asymptotics as $|x|\to \infty$ of the hypergeometric functions can be found in \cite{sheng2020fast} and are given by
\begin{align}
     \hgeom{1}{1}\left(s+\frac{d}{2}, \frac{d}{2}, -|x|^2\right) &= \frac{C(d,s)}{|x|^{d+2s}} (1+ O(|x|^{-2}))\nonumber\\
     \hgeom{2}{1}\left(s+\alpha, s+\frac{d}{2}, \frac{d}{2}, -|x|^2\right) &= \frac{C(d,s)}{|x|^{\min(2\alpha,d)+2s}} (1+ o(1)) \quad. \label{hypergeom2}
\end{align}
This decay induces $L^{2,\delta}$ functions at infinity for $\delta$ such that $2(d+2s-\delta) > d$. Since neither of the functions have singularities, they will be in $L^{2, \delta}(\R^d)$ for $\delta \in (\frac{1}{2}, 1)$ in dimension $d=2$. \\

\begin{figure}[h!]
\centering
    \includegraphics[scale=0.5]{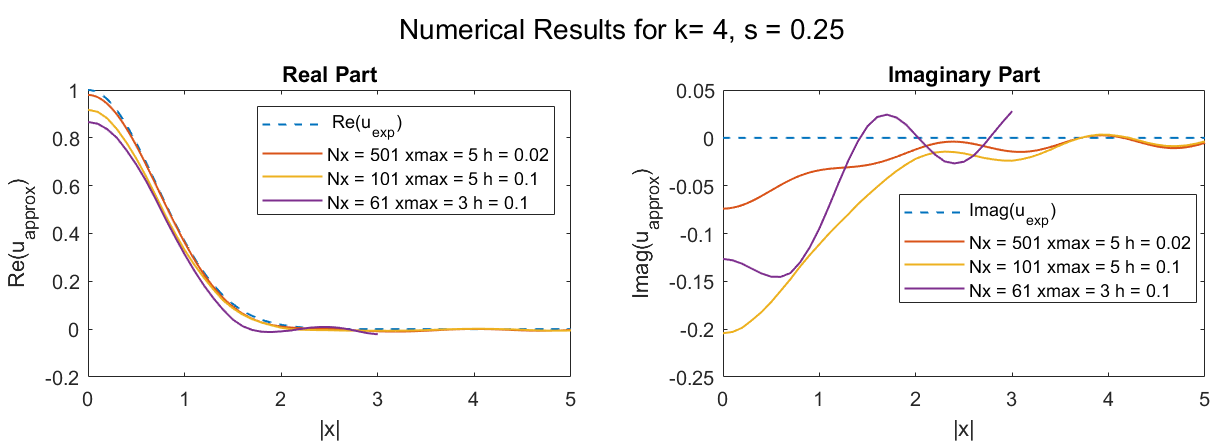}\\
    \centering
    \includegraphics[scale=0.5]{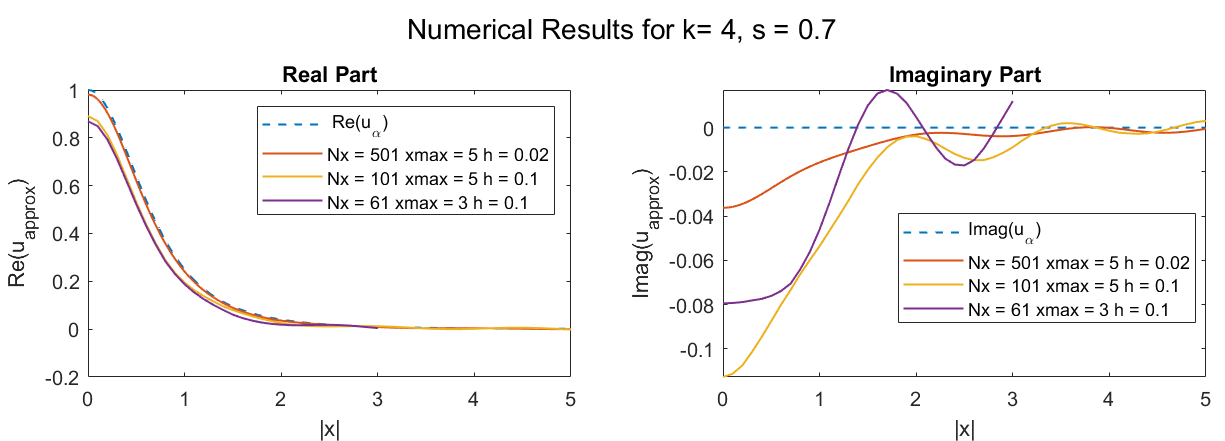}
    \caption{Qualitative Result for the test right hand sides $f_{exp}$ and $f_\alpha$ for $\alpha=2$ in 2D}
    \label{fig:qualitative_res}
\end{figure}

\indent The qualitative results for the test right hand sides $f_{exp}$ and $f_{\alpha}$ are shown in figure \ref{fig:qualitative_res}. The dashed lines correspond to the exact solutions, that is $u_{\exp}$ in the top figures and $u_{\alpha}$ in the bottom ones. The solid lines correspond to the numerical convolution of $f_{\exp}$ and $f_{\alpha}$ with the fundamental solution computed as described above, for different sets of parameters varying the mesh size ($h$) and the size of the domain ($x_{max}$). We see that the finer the discretization is (i.e. the smaller $h$ is) and the larger $x_{max}$ is, the more accurate the results are, as expected. Note that taking $h \to 0$ is not enough to have convergence, as we also need $x_{max} \to \infty$ to account for the truncation of the integral. For a more quantitative perspective, figure (\ref{error}) displays the error $\| u_{approx} - u_{\exp}\|$, both in $L^2$ and $L^\infty$ norms, as the mesh size is refined for $k=4$, $s=0.7$ and $x_{max}= 5$.
\begin{figure}[h!]
    \centering
\includegraphics[width=0.8\linewidth]{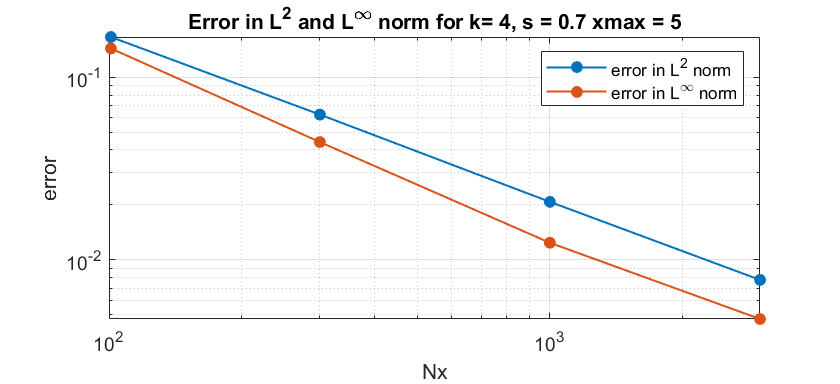}
    \caption{Error plot for mesh refinement}
    \label{error}
\end{figure}
\newpage
\begin{remark} It is interesting to test this numerical convolution with right hand sides that do not belong to $L^{2,\delta}(\R^d)$ for any $\delta \in (\frac{1}{2},1)$. Theoretically, we no longer have the uniqueness of the solution to the fractional Helmholtz equation and the convolution with the fundamental solution $\Phi_{s,k}$ is a priori not well defined outside of $L^{2,\delta}$, $\delta \in (\frac{1}{2}, 1)$. However, we can compute the convolution numerically. We can take for example the previous right hand side $f_{\alpha}$ with $\alpha< \frac{3}{4}$, so that $f_{\alpha} \not \in L^{2,\delta}(\R^2)$ for any $\delta \in (\frac{1}{2},1)$. The above calculations remain valid, that is $(\FLs - k^{2s}) u_{\alpha} = f_{\alpha}$. Figure (\ref{small_alpha}) shows the numerical result. It seems that the convolution does approximate the expected solution, with some oscillation phenomena. Note that since the function decays slower, we consider $x_{max}$ larger than before for more precision. 

\begin{figure}[h!]
    \centering
\includegraphics[width=\linewidth]{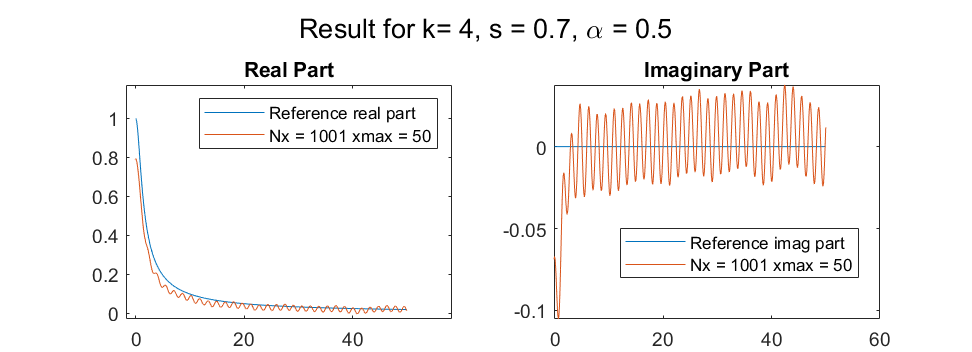}
    \caption{Numerical result for $\alpha = 0.5$}
    \label{small_alpha}
\end{figure}
\end{remark}

 \subsection{Inverse Problem Reconstruction (Factorization Method)}

To compute a numerical approximation of the far field operator, we discretize the sphere into $N_{inc}$ angles $\{\theta_j\}$ equally spaced. We then generate a square matrix $F$ of size $N_{inc} \times N_{inc}$ which approximate the far field operator, and whose entries $F_{ij}$ correspond to the far field pattern $u^\infty( \theta_i; \theta_j)$ measured at angle $\theta_i$ for an incident plane wave sent with angle $\theta_j$. Recall that the far field pattern is given by formula (\ref{far-field}), and combined with (\ref{LSh}) we have
$$
u^\infty(\hat x; \theta) =\frac{k^2}{s}\int_D e^{- i k \hat x\cdot y}(n(y)-1)((I-T_{s,k})^{-1}  e^{i k \langle\theta, \cdot \rangle }) (y) dy.
$$
We numerically approximate the above formula by
\begin{align*}
    F_{ij} =[ Q(I-M)^{-1}e^{i k \langle\theta_j, \cdot \rangle }]_i
    \end{align*}
where $M$ is of size $N_{supp} \times N_{supp}$, with $N_{supp}$ being the size of the support of the inhomogeneity, and $Q$ is of size $N_{inc} \times N_{supp}$ such that
\begin{align*}
    M_{ij} &:=\begin{cases}
k^{2s}\Phi^{approx}_{s,k}(x_i-x_j) (n-1)(x_j) h^d\quad & \mbox{ if } i\not = j \\
        k^{2s} (n-1)(x_i) I\Phi^{approx}\quad & \mbox{ if } i = j  
    \end{cases}   ,\quad  x_i, x_j \in \mbox{supp} (n-1),\\
    Q_{ij} &:=
   \frac{k^{2}}{s} e^{- i k \theta_i \cdot x_j} (n-1)(x_j) h^d,\quad  x_j \in \mbox{supp} (n-1). 
\end{align*}
 
 Theorem \ref{FM_main} states that $z\in D$ if and only if there exists $g\in L^2(\mathbb{S}^{d-1})$ such that $(F^*F)^{1/4} g = \Phi^\infty(\cdot, z)$. We numerically solve this equation, and the numerical solution is expected to have a large $L^2$ norm when $z\not \in D$. Multiplying the equation by $(F^*F)^{1/4}$, it is equivalent to solving  
\begin{align}\label{FM_eq}
   (F^*F)^{1/2} g = (F^*F)^{1/4} \Phi^\infty(\cdot, z).
\end{align}
Numerically, the notation $\Phi^\infty(\cdot, z)$ corresponds to a vector of size $N_{inc}$ with entries $\Phi^\infty(\theta_j, z), j \in \{1,..., N_{inc}\}$. Let $F = U S V^T$ be the singular value decomposition of the matrix $F$, then
\begin{align*}
    F^*F =\overline{V}|S|^2 V^T
\end{align*}
where $|S|$ denotes the diagonal matrix of the moduli $|\lambda_j|$ of the eigenvalues $\{\lambda_j\}_{1\leq j \leq N_{inc}}$ of $F$, and equation (\ref{FM_eq}) becomes 
$$
V^Tg = |S|^{-1/2} V^T\Phi^\infty(\cdot, z).
$$
Computing the $L^2$-norm of the vector $V^Tg$ gives an indicator type quantity that becomes large for $z$ outside of $D$, or equivalently the inverse of the $L^2$-norm of $V^T g$ given by
$$
W :=\left(\sum_{j=1}^{N_{inc}} \frac{[V^T\Phi^\infty(\cdot, z)]^2_j}{|\lambda_j|} \right)^{-1} 
$$
becomes small when $z\not \in D$. 
In what follows, we plot the quantity $W$ for different choices of parameters. We take the number of discretization points to be $N_x = 400$, and the size of the square to be $x_{max} = 5$.

In figure (\ref{fig:cercle+rectangle}), we compare the reconstructions for different values of $s$ at the frequency $k=5$ and the number of incident angles $N_{inc}= 72$. The top left picture is the true scatterer. The reconstructions capture the presence of the two distinct scatterers independently of $s=0.2$ or $s=0.7$. However, the reconstruction seems less precise as $s$ is larger, and even less so for the bottom right picture which corresponds to $s=1$, the Helmholtz case. This might be a manifestation of the better stability often induced by non-local operators. 
\begin{figure}[p]
    \centering
\includegraphics[scale=0.6]{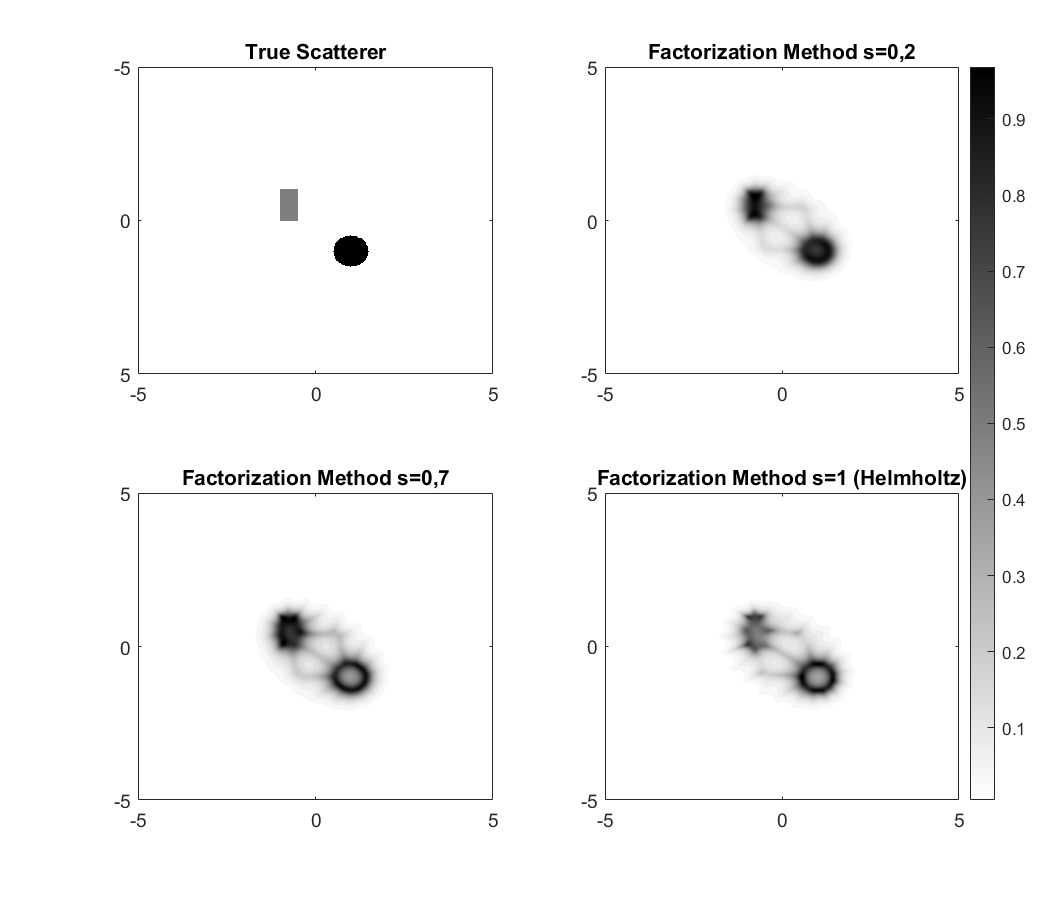}
    \caption{Reconstruction of the support of an inhomogeneity for different $s$}
\label{fig:cercle+rectangle}
\end{figure}

We then test the Factorization Method algorithm on a non-convex boomerang-like shape (first plot of figure (\ref{boomerang}) ) with two different frequencies and two different numbers of incident waves for an order $s=0.6$. At a higher frequency (bottom left), the reconstruction is sharper and better localized than at a lower frequency (top right). However the concavity is not as well recovered for the higher frequency $k=5$, which might be due to less stability in the numerical computations for higher frequencies. A larger number of incident fields does not seem to affect the results in this case, as we can see between the bottom left ($N_{inc} = 36$) and bottom right ($N_{inc} = 72$) figures.

\begin{figure}[t]
    \centering
    \includegraphics[scale=0.6]{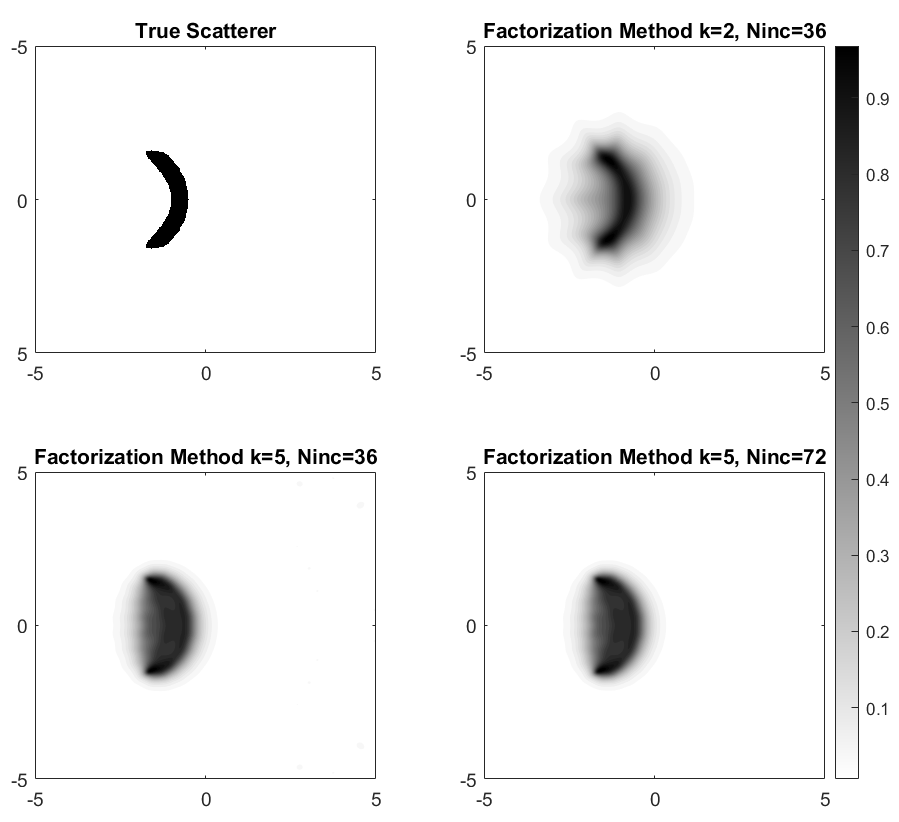}
    \caption{Reconstruction at different frequencies and different numbers of incident waves ($N_{inc}$)}
    \label{boomerang}
\end{figure}
\newpage


\section*{Acknowledgements}
This work was partially supported by NSF grant DMS-24-06313. The author would like to thank Professors F. Cakoni and M. S. Vogelius for valuable discussions on this project. 
\section*{Data Accessibility}
The code used in this paper is openly available with explanations on Zenodo:  \url{https://doi.org/10.5281/zenodo.19635828} \cite{dana} and GitHub: \url{https://github.com/dana-zbbg/FactorizationMethod_FractionalHelmholtz/releases/tag/v1.2}.

\section{Appendix}
 \subsection{Asymptotic Behavior of the Scattered Field}
 Let us denote the radiating fundamental solution of the fractional Helmholtz equation by  $\Phi_{s,k}$ and the Helmholtz fundamental solution by $\Phi_{helm,k}$. In \cite{zilberberg2026limiting}, it is proved that $\Phi_{s,k}$ in dimension $d=1,2,3$ for $s\in (0,1)$ is given by a multiple of the fundamental solution for Helmholtz equation and a term which decays faster at infinity. Let us denote this remainder term by
$$
\Phi_{s,k}^{\Delta} := \Phi_{s,k} - \frac{k^{2-2s}}{s} \Phi_{helm,k}.
$$
Recall that the fundamental solution for Helmholtz equation is given by
\begin{align*}
    \Phi_{helm,k}(x) = \begin{cases}
        \frac{e^{ik|x|}}{k} \quad & d=1,\\
        \frac{i}{4} H^{(1)}_0(k|x|)  \quad &d=2,\\
        \frac{e^{ik|x|}}{4\pi |x|} \quad & d=3\quad 
    \end{cases}\quad
\end{align*}
and has the asymptotic expansion 
\begin{align}\label{asymotic_fund_sol_Helm}
    \Phi_{helm,k}(x)
    &= c_d k^{\frac{d-3}{2}}\frac{e^{ik|x|}}{(\pi|x|)^{\frac{d-1}{2}}} + O_{|x|\to \infty}(|x|^{\frac{-d-1}{2}})\quad 
\end{align}
where $c_1 = 1, c_2 = \frac{e^{-i\pi/4}}{2\sqrt{2}}, c_3 = \frac{1}{4}$. According to the computations done in \cite{zilberberg2026limiting}, we have 
\begin{numcases}{\Phi_{s,k}^\Delta (x) = }
    \frac{k^{1-2s}}{\pi } \int_0^\infty \frac{e^{-yk|x| }y^{2s} \sin(s\pi)}{y^{4s}+ 1 - 2\cos(s\pi) y^{2s}} & \quad  d=1, \label{Phi_sk_delta1}\\
      \sum_{j=0}^{m-1} \frac{c_{2,j} k^{2sj}}{|x|^{2-2s(j+1)}} + \frac{1}{2\pi} \int_0^\infty J_0(r|x|) F(r,k) rdr + L(|x|, k) & \quad d=2,\label{Phi_sk_delta2}\\
        \sum_{j=0}^{m-1} \frac{c_{3,j}k^{2sj}}{|x|^{3-2s(j+1)}} + \frac{k^{2-2s}}{2\pi^2 |x|} \int_0^\infty e^{-k|x|y} y^{1-2sm} \frac{y^{2s}\sin(\pi s(m+1)) -|x|^{2s}\sin(\pi sm)}{y^{4s} + 1 - 2 \cos(s\pi) y^{2s}} dy& \quad d=3,\label{Phi_sk_delta3}
    \end{numcases}
where $m:= \left\lfloor \frac{1}{2s}\right\rfloor$,
\begin{align*}
&F(r,k) :=\begin{cases}
 \frac{k(r^2-k^{2}) - s^{-1}k^{2-2s}r(r^{2s}-k^{2s}) + k^{2-2s}(r^{2s}-k^{2s})(r-k)}{r(r^{2s}-k^{2s})(r^2-k^2)} &\mbox{ if } s\in (0,\frac{1}{2}], \frac{1}{2s} \in \mathbb{N},\\
   \frac{k^{2sm}(r^2-k^{2}) - s^{-1}k^{2-2s}r^{2sm}(r^{2s}-k^{2s})}{r^{2sm}(r^{2s}-k^{2s})(r^2-k^2)}&\mbox{ otherwise },
   \end{cases}\\
  &L(r,k) :=\begin{cases}
   - \frac{k^{2-2s}}{4} K_0(k|x|)&\mbox{ if } s\in (0,\frac{1}{2}], \frac{1}{2s} \in \mathbb{N},\\
 0 &\mbox{ otherwise }, 
   \end{cases} \\
      & c_{d,j} := \frac{\Gamma(d/2 - s(j+1))}{4^{s(j+1)}\pi^{d/2} \Gamma(s(j+1))}, \quad 
\end{align*}
and $K_0$ is the Struve function of the second kind of order $0$. We take the convention $\sum_{j=0}^{-1} = 0$. 

The following lemma follows from the analysis done in \cite{zilberberg2026limiting}. 
\begin{lemma}\label{lemma_Phi_delta}
    \begin{itemize} 
        \item[(i)] The asymptotic behavior of $\Phi_{s,k}$ at infinity is the same as $\frac{k^{2-2s}}{s}\Phi_{helm,k}$ and
        \begin{align*}
           \left| \Phi_{s,k}(x) - \frac{k^{2-2s}}{s} \Phi_{helm,k}(x) \right| = \left| \Phi_{s,k}^\Delta(x)\right| = O(|x|^{-\frac{d}{2}}).
        \end{align*}
        \item[(ii)]   $\Phi^\Delta_{s,k}$ is real valued, its convolution with $L^2(D)$ functions defines a compact operator from $L^2(D) \to L^2(D)$ and it satisfies the estimate 
        \begin{align*}
            \| \Phi_{s,k}*f\|_{L^2(D)} \leq C(D,s) (1+k^{2sm} + k^{-s} + k^{1-2s})\|f\|_{L^2(D)}, \quad \mbox{ for all } f \in L^2(D).
        \end{align*}
    \end{itemize}
\end{lemma}
\begin{proof}
    \begin{itemize}
        \item[(i)] The fundamental solution for Helmholtz equation behaves at infinity like $O(|x|^{\frac{-d+1}{2}})$. Theorem 5.1 in appendix 5.A of \cite{zilberberg2026limiting} shows that $\Phi_{s,k}^\Delta(x)$ is at least a $O(|x|^{-\frac{d}{2}}) $, hence the result. 
        \item[(ii)] We can see from the expression of $\Phi_{s,k}^\Delta$ given above by (\ref{Phi_sk_delta1}), (\ref{Phi_sk_delta2}) and (\ref{Phi_sk_delta3}) that it is real valued. Compactness follows from the embedding $H^{2s}(D) \hookrightarrow L^2(D)$. To prove the norm estimate, we analyze each dimension and show that $\Phi_{s,k}\in L^1(D)$. In dimension 1, when $s>1/2$, we can bound below the denominator in the integral (\ref{Phi_sk_delta1}) by $y^{4s}+1$, and bound the exponential by $1$. We obtain $|\Phi_{s,k}^{\Delta}(x)|\leq C(s) k^{1-2s}$. When $s<1/2$, we can bound the 
        denominator below by $(1-\cos^2(s\pi) )y^{4s}$ and after the change of variable $\tilde y = k |x|y$ we obtain $|\Phi_{s,k}^{\Delta}(x)|\leq C(s) |x|^{2s-1}$. When $s=1/2$, we can directly perform an integration by parts 
        \begin{align*}
            \Phi_{s,k}^{\Delta}(x) &= \frac{1}{\pi}\int_0^\infty \frac{e^{-y |x|k}y}{y^2 + 1} dy =  \frac{1}{2\pi} \int_0^\infty k|x|e^{-y |x|k } \ln(1+ y^2) dy\\
            &=  \frac{1}{2\pi} \int_0^\infty e^{-y} \ln\left( 1+ \frac{y^2}{k^2 |x|^2}\right) dy\\
            &\leq \frac{2}{\pi k^{1/2}|x|^{1/2}} \int_0^\infty e^{-y} y^{1/2}  dy
        \end{align*}
        where we used $\ln( 1+ z) \leq 4z^{1/4}$. 
        We deduce that there exists some constant $C(D,s)$ depending only on $s$ and $D$ such that 
        $$
\|\Phi_{s,k}^\Delta\|_{L^1(D)}\leq C(D,s) (1+k^{-s}+k^{1-2s}) , \quad d=1. 
        $$
In dimension 2, following the arguments of lemma 5.3 in appendix 5A of \cite{zilberberg2026limiting}, we have 
        \begin{align*}
            | \Phi_{s,k}^\Delta(x)| \leq \sum_{j=0}^{m-1} \frac{c_{2,j}k^{2sj}}{|x|^{2-2s(j+1)}} + \frac{\| \nabla F(\cdot , k) \|_{L^1(\R^2)}}{|x|} + |L(|x|,k)|
        \end{align*}
        where again $m= \lfloor\frac{1}{2s}\rfloor$, and we take the convention $\sum_{j=0}^{-1} = 0$. We observe, still from the computations of  lemma 5.3 of Appendix 5A of \cite{zilberberg2026limiting}, that 
        $$
       \frac{\partial}{\partial r} F(r,k) = k^{-1-2s} \frac{\partial}{\partial r}F\left(\frac{r}{k},1\right)\quad,
        $$
        hence after a change of variable $\| \nabla F(\cdot, k)\|_{L^1(\R^2)} = k^{1-2s} \| \nabla F(\cdot, 1)\|_{L^1(\R^2)}$. We then use that for some constant $C$ depending on $D$, $\frac{k^{2-2s}}{4}|K_0(k|x|)| \leq \frac{C k^{1-2s}}{ |x|}$ for all $x\in D$, and that $k^{2s(m-1)} \leq 1 + k^{1-2s}$ for $m\geq 1$. We obtain 
        $$
       \| \Phi_{s,k}^\Delta \|_{L^1(D)} \leq C(D,s) (1 + k^{1-2s}), \quad d=2
        $$
        for some constant $C(D,s)$ depending only on $D$ and $s$. In dimension 3, we use the following bound on the integral part of $\Phi_{s,k}^\Delta$
        \begin{align*}
           & \frac{k^{2-2s}}{2\pi^2 |x|} \left|\int_0^\infty e^{-k|x|y} \frac{y^{1-2sm} y^{2s} \sin(\pi s(m+1))}{y^{4s}+1 -2\cos(s\pi) y^{2s}} dy -  \int_0^\infty e^{-k|x|y} \frac{y^{1-2sm} |x|^{2s} \sin(\pi sm)}{y^{4s}+1 -2\cos(s\pi) y^{2s}} dy\right|\\
            &\leq  \frac{k^{2-2s}}{2\pi^2 |x|} \left(\int_0^\infty \frac{e^{-k|x|y} y^{1-2s(m+1)} \sin(\pi s (m+1))}{(1-\cos^2(s\pi))} dy + |D|^{2s}\int_0^\infty \frac{e^{-k|x|y} y^{1-2s(m+1)} \sin(\pi sm)}{2(1-\cos(s\pi))} dy  \right)  \\
            &\leq C(D,s) \frac{k^{2sm}}{|x|^{3-2s(m+1)}} 
        \end{align*}
        where for the first inequality, we bounded below the denominator by $(1-\cos^2(s\pi) )y^{4s}$ in the first integral and by $2(1-\cos(s\pi) ) y^{2s}$ in the second integral. In the last inequality, we performed the change of variable $\tilde y = k|x|y$ to obtain the bound. We deduce that for some constant depending only on $s$ and $D$  
        $$
\|\Phi_{s,k}^\Delta\|_{L^1(D)}\leq C(s,D) (1+k^{2sm}) ,\quad d=3.
        $$
        We conclude that in dimension 1, 2, 3, we have the estimate
        $$
       \|\Phi_{s,k}^\Delta \|_{L^1(D)} \leq C(D,s) (1+k^{2sm} + k^{-s} + k^{1-2s}) .
        $$
        Young's inequality for convolution yields the result. 
    \end{itemize}
\end{proof}

\bibliographystyle{alpha}
\bibliography{Corrections/sample}

\end{document}